\documentclass{amsart}
\usepackage[style=alphabetic]{biblatex} 
\usepackage{tikz}
\usepackage{amssymb}
\usepackage{amsthm}
\usepackage[all]{xy}
\usepackage{microtype}
\usepackage{xcolor}
\usepackage{hyperref}
\usepackage[nameinlink]{cleveref}

\hypersetup{
 colorlinks,
 linkcolor={teal},
 citecolor={teal},
 urlcolor={teal}
}

\calclayout

\DeclareFieldFormat
  [article,book,inbook,incollection,inproceedings,patent,thesis,unpublished]
  {title}{\emph{#1\isdot}}

\theoremstyle{plain}
	\newtheorem{theorem}{Theorem}[section]
	\newtheorem{lemma}[theorem]{Lemma}
	\newtheorem{corollary}[theorem]{Corollary}
	
	\newtheorem{proposition}[theorem]{Proposition}
	\newtheorem{question}[theorem]{Question}

\theoremstyle{definition} 
	
	\newtheorem{definition}[theorem]{Definition}

\DeclareMathOperator{\Nm}{Nm}
\def\jac#1#2{\left(\frac{#1}{#2}\right)}
\def\D{\mathrm{D}}
\def\id{\mathrm{id}}
\def\sgn{\mathop{\mathrm{sgn}}}

\def\Cl{\mathop{\mathrm{Cl}}}

\begin{document}
\title{Quadratic residues and domino tilings}
\author[Y. Kamio]{Yuhi Kamio}
\address{College of Arts and Sciences, University of Tokyo, 3-8-1 Komaba, Meguro-ku, Tokyo 153-8914, Japan}
\email{emirp13@g.ecc.u-tokyo.ac.jp}
\author[J. Koizumi]{Junnosuke Koizumi}
\address{Graduate School of Mathematical Sciences, University of Tokyo, 3-8-1 Komaba, Meguro-ku, Tokyo 153-8914, Japan}
\email{jkoizumi@ms.u-tokyo.ac.jp}
\author[T. Nakazawa]{Toshihiko Nakazawa}
\address{KADOKAWA DWANGO Educational Institute, Kabuki-za Tower 14F, 4-12-15 Ginza, Chuo-ku, Tokyo 104-0061, Japan}
\email{toshihiko\_nakazawa@nnn.ac.jp}

\date{\today}
\thanks{}

\begin{abstract}
    The formula for the number of domino tilings due to Kasteleyn and Temperley-Fisher is strikingly similar to Eisenstein's formula for the Legendre symbol.
    We study the connection between these two concepts and prove a formula which expresses the Jacobi symbol in terms of domino tilings.
\end{abstract}

\maketitle
\setcounter{tocdepth}{1}
\tableofcontents

\enlargethispage*{20pt}
\thispagestyle{empty}

\section{Introduction}

In 1848, Eisenstein \cite{Eisenstein1848} presented a proof of the quadratic, cubic, and quartic reciprocity laws using analytic functions.
The idea of his proof of the quadratic reciprocity was to express the Legendre symbol using trigonometric functions.
For odd prime numbers $p$ and $q$, his formula can be written as follows:
$$
    \jac{q}{p}=4^{\frac{p-1}{2}\frac{q-1}{2}}\prod_{j=1}^{\frac{p-1}{2}}\prod_{k=1}^{\frac{q-1}{2}}\biggl(\cos^2 \dfrac{2\pi j}{p} - \cos^2 \dfrac{2\pi k}{q}\biggr).
$$
About 100 years later, in 1961, Kasteleyn \cite{Kasteleyn_domino} and Temperley-Fisher \cite{Temperley_Fisher_domino} independently discovered a formula for the number of domino tilings of a rectangular domain.
For odd numbers $m$ and $n$, the number of domino tilings of a rectangle with width $m-1$ and height $n-1$ is given by the following formula:
$$
    k_{m-1,n-1}=4^{\frac{m-1}{2}\frac{n-1}{2}}\prod_{j=1}^{\frac{m-1}{2}}\prod_{k=1}^{\frac{n-1}{2}}\biggl(\cos^2 \dfrac{2\pi j}{m} + \cos^2 \dfrac{2\pi k}{n}\biggr).
$$
These formulas are strikingly similar, yet there seems to be no literature that addresses this similarity.
The purpose of this paper is to reveal the hidden connection between quadratic residues and domino tilings.

For non-negative integers $m,n$, we write $R_{m,n}$ for the rectangle with width $m$ and height $n$, and $\D(R_{m,n})$ for the set of all domino tilings of $R_{m,n}$.
For $D\in \D(R_{m,n})$, we write $h(D)$ for the number of horizontal dominoes.
Our main theorem is the following.

\begin{theorem}\label{main}
    Let $m,n$ be positive integers and assume that $n$ is odd.
    Then we have
    \begin{align}\label{eq:main}
        \sum_{D\in \D(R_{m-1,n-1})} \sqrt{-1}^{h(D)}=
        \begin{cases}
            \jac{m}{n}&(m\equiv 1\pmod 2)\\
            \jac{m/2}{n}&(m\equiv 0\pmod 2).
        \end{cases}
    \end{align}
    Here, the right hand side is the Jacobi symbol.
\end{theorem}

In particular, this formula implies that the left hand side takes values in $\{-1,0,1\}$, which is already non-trivial.
Also, this formula immediately implies the quadratic reciprocity as follows.
For odd numbers $m$ and $n$, there is an obvious bijection $\rho\colon \D(R_{m-1,n-1})\xrightarrow{\sim} \D(R_{n-1,m-1})$ induced by transposition.
Since the total number of dominoes is $\frac{(m-1)(n-1)}{2}$, we have $h(\rho(D))=\frac{(m-1)(n-1)}{2}-h(D)$.
Therefore \Cref{eq:main} implies
$$
    \jac{m}{n}=(-1)^{\frac{m-1}{2}\frac{n-1}{2}}\jac{n}{m}.
$$

In this paper, we provide two different proofs of \Cref{main}.
The first one is an algebraic proof based on the method of Kasteleyn.
The second one is a completely elementary proof, which yields a new combinatorial proof of the quadratic reciprocity.

\begin{question}
    Is there a formula similar to \Cref{main} for higher residue symbols?
\end{question}

\subsection*{Acknowledgement}
The authors are grateful to Ryuya Hora, Kyosuke Higashida, Shuho Kanda, and Koto Imai for stimulating discussions at “Mathspace Topos”, and would like to extend our gratitude to its promotor Nobuo Kawakami for his support in establishing the “Mathspace Topos” and adviser Fumiharu Kato for fostering such an environment.
The authors would also like to thank Yugo Takanashi and Shin-ichiro Seki for helpful comments.

\section{Preliminaries}

\subsection{Domino tilings}

We define a \emph{board} to be a finite subset of $\mathbb{Z}\times \mathbb{Z}$.
For example, $R_{m,n}=[m]\times [n]$, where $[n]=\{1,2,\dots,n\}$, is a board.
An element of a board is called a \emph{square}.
A \emph{domino tiling} of a board $X$ is a set $D$ of subsets of $X$ such that
\begin{itemize}
    \item every element of $D$ consists of two adjacent squares in $X$, and
    \item every square of $X$ is contained in exactly one element of $D$.
\end{itemize}
We write $\D(X)$ for the set of all domino tilings of $X$.
Each element of $D\in \D(X)$ is called a \emph{domino}.
A domino is called \emph{horizontal} if it is of the form $\{(i,j),(i+1,j)\}$, and is called \emph{vertical} if it is of the form $\{(i,j),(i,j+1)\}$.
A domino tiling $D$ is called \emph{totally vertical} if $D$ does not contain a horizontal domino.

We draw $X$ as a collection of $1\times 1$ squares centered at elements of $X$, and $D$ as a collection of $1\times 2$ rectangles formed by two adjacent squares in $X$.
For example, the domino tiling
$$
    D=\bigl\{\{(1,1),(2,1)\},\{(1,2),(2,2)\},\{(3,1),(3,2)\}\bigr\}
$$
of $R_{3,2}$ can be drawn as follows.

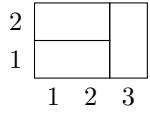
\begin{center}
    \begin{tikzpicture}[scale=0.5]
    \draw (0,0) rectangle (2,1);
    \draw (0,1) rectangle (2,2);
    \draw (2,0) rectangle (3,2);
    \node at (0.5,-0.5) {$1$};
    \node at (1.5,-0.5) {$2$};
    \node at (2.5,-0.5) {$3$};
    \node at (-0.5,0.5) {$1$};
    \node at (-0.5,1.5) {$2$};
    \end{tikzpicture}
\end{center}

For a domino tiling having a $2\times 2$ square covered by two dominoes, we can rotate the square by $90^\circ$ to obtain another domino tiling.
This procedure is called a \emph{flip}.

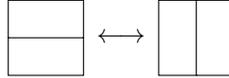
\begin{center}
    \begin{tikzpicture}[scale=0.5]
        \draw (0,0) -- (0,1) -- (2,1) -- (2,0) -- (0,0) -- (0,-1) -- (2,-1) -- (2,0);
        \node at (3,0) {$\longleftrightarrow$};
        \draw (5,1) -- (4,1) -- (4,-1) -- (5,-1) -- (5,1) -- (6,1) -- (6,-1) -- (5,-1);
    \end{tikzpicture}
\end{center}
The following lemma is well-known, but we provide a proof for the convenience of the reader.

\begin{lemma}\label{flip_transitivity}
    Let $m,n$ be poisitive integers with $n$ even.
    Then any domino tiling of $R_{m,n}$ can be obtained from the totally vertical tiling by a finite sequence of flips.
\end{lemma}

\begin{proof}
    Let $D$ be a domino tiling of $R_{m,n}$.
    We present an algorithm to transform $D$ into the totally vertical tiling by a finite sequence of flips.
    First we transform $D$ so that the following condition is satisfied:
    \begin{quote}
        $(\heartsuit)$ $(1,1)$ and $(2,1)$ are covered by vertical dominoes.
    \end{quote}
    Suppose that the square $(1,1)$ is covered by a horizontal domino.
    We define squares $q_1,q_2,\cdots$ by
    $$
        q_k = (\lceil k/2 \rceil,\lceil (k+1)/2 \rceil).
    $$
    Take the smallest $k$ such that the dominoes covering $q_k$ and $q_{k+1}$ are both horizontal or both vertical.
    Then performing flips for $q_k+\{0,1\}^2,q_{k-1}+\{0,1\}^2,\dots,q_1+\{0,1\}^2$ in this order, we can achieve $(\heartsuit)$.
    Similarly we can achieve $(\heartsuit)$ if the square $(1,1)$ is covered by a vertical domino.

    
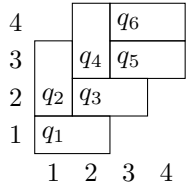
\begin{center}
        \begin{tikzpicture}[scale=0.5]
        \newcommand{\horizontal}{
            \draw (0,0) rectangle (2,1);
        }
        \newcommand{\vertical}{
            \draw (0,0) rectangle (1,2);
        }
        \node foreach \x in {1,2,3,4} at (\x-0.5, -0.5) {$\x$};
        \node foreach \x in {1,2,3,4} at (-0.5, \x-0.5) {$\x$};
        \begin{scope}[shift={(0,0)}]
            \horizontal
            \node at (0.5,0.5) {$q_1$};
        \end{scope}
        \begin{scope}[shift={(0,1)}]
            \vertical
            \node at (0.5,0.5) {$q_2$};
        \end{scope}
        \begin{scope}[shift={(1,1)}]
            \horizontal
            \node at (0.5,0.5) {$q_3$};
        \end{scope}
        \begin{scope}[shift={(1,2)}]
            \vertical
            \node at (0.5,0.5) {$q_4$};
        \end{scope}
        \begin{scope}[shift={(2,2)}]
            \horizontal
            \node at (0.5,0.5) {$q_5$};
        \end{scope}
        \begin{scope}[shift={(2,3)}]
            \horizontal
            \node at (0.5,0.5) {$q_6$};
        \end{scope}
        \end{tikzpicture}
    \end{center}
    
    The same argument shows that we can further transform $D$ so that $(1,3)$ and $(2,3)$ are covered by vertical dominoes.
    Repeating this procedure, we can transform $D$ into the totally vertical tiling.
\end{proof}

\begin{lemma}\label{h(D)_even}
    Let $m,n$ be poisitive integers with $n$ even.
    For any domino tiling $D$ of an $R_{m,n}$, the number of horizontal dominoes $h(D)$ is even.
\end{lemma}

\begin{proof}
    The claim is clear for the totally vertical tiling.
    Since a flip does not change the parity of $h(D)$, the claim is true for any domino tiling by \Cref{flip_transitivity}.
\end{proof}

\subsection{Norms}
    Let $A$ be a $\mathbb{C}$-algebra which is finite-dimensional as a $\mathbb{C}$-vector space.
    For $a\in A$, the norm $\Nm_{A/\mathbb{C}}(a)$ is defined to be the determinant of the multiplication map $\times a\colon A\to A$, viewed as a $\mathbb{C}$-linear map.

\begin{lemma}\label{norm_product}
    Let $P(T),Q(T)\in \mathbb{C}[T]$ be separable monic polynomials.
    Set $A=\mathbb{C}[T_1,T_2]/(P(T_1),Q(T_2))$ and write $t_i$ for the image of $T_i$ in $A$.
    Let $\xi_1,\dots,\xi_d$ (resp. $\eta_1,\dots,\eta_e$) be the roots of $P(T)$ (resp. $Q(T)$) in $\mathbb{C}$.
    Then for any $f(T_1,T_2)\in \mathbb{C}[T_1,T_2]$, we have
    $$
        \Nm_{A/\mathbb{C}}(f(t_1,t_2))=\prod_{i=1}^d \prod_{j=1}^e f(\xi_i,\eta_j).
    $$
\end{lemma}

\begin{proof}
    By the Chinese remainder theorem, we have an isomorphism of $\mathbb{C}$-algebras
    $$
        A\cong \prod_{i=1}^d\prod_{j=1}^e \mathbb{C}[T_1,T_2]/(T_1-\xi_i,T_2-\eta_j)
        \cong \prod_{i=1}^d\prod_{j=1}^e \mathbb{C};\quad f(t_1,t_2)\mapsto f(\xi_i,\eta_j).
    $$
    The claim is clear from this isomorphism.
\end{proof}

\section{Algebraic Proof}

\subsection{Kasteleyn operator}
Following Kasteleyn's method \cite{Kasteleyn_domino}, we express the left hand side of \Cref{eq:main} as the determinant of a certain linear operator, which we call the Kasteleyn operator.

Let $m,n$ be positive integers and assume that $n$ is odd.
Let $F$ denote the $\mathbb{C}$-vector space freely generated by the squares in $R_{m-1,n-1}$.
We define $V$ to be the quotient of $F$ by relations $(i,j)=-(i,n-j)$, and write $e_{i,j}$ for the image of $(i,j)$ in $V$.
Then we have
$$
    \dim_\mathbb{C} V = \dfrac{(m-1)(n-1)}{2},
$$
and the set $\{e_{i,j}\mid i+j\equiv \mu\pmod 2\}$ is a $\mathbb{C}$-basis of $V$ for each $\mu = 0,1$.
\begin{definition}
    We define the \emph{Kasteleyn operator} $K\colon V\to V$ by
    $$
    K(e_{i,j}) = e_{i-1,j} + e_{i+1,j} + e_{i,j-1} + e_{i,j+1} \quad (i+j\equiv 0\mod 2).
    $$
Here, we set $e_{i,j}=0$ if $i\in \{0,m\}$ or $j\in \{0,n\}$.
\end{definition}

\begin{lemma}\label{kasteleyn_det}
    The following equality holds:
    $$
        \sum_{D\in \D(R_{m-1,n-1})} \sqrt{-1}^{h(D)} =
        \begin{cases}
            \det K & (m\equiv 1 \pmod 2)\\
            (-1)^{\frac{n^2-1}{8}}\cdot \det K & (m\equiv 0 \pmod 2).
        \end{cases}
    $$
\end{lemma}

\begin{proof}
    For $\mu\in \{0,1\}$, we define a set $T_\mu$ as follows:
    $$
        T_\mu = \{(i,j)\mid 1\leq i\leq m-1,\;1\leq j\leq n-1,\;i+j\equiv \mu \pmod 2\}.
    $$
    Then we have canonical identifications $V\cong \mathbb{C}^{\oplus T_0}\cong \mathbb{C}^{\oplus T_1}$.
    We say that a map $\varphi\colon T_0\to T_1$ is \emph{admissible} if for any $(i,j)\in T_0$ we have
    $$
        \varphi(i,j)\in \{(i-1,j), (i+1,j), (i,j-1), (i,j+1)\}.
    $$
    By the multilinearity of the determinant, we have
    $$
        \det K = \sum_{\substack{\varphi\colon T_0\to T_1\\ \text{admissible}}}\det \varphi_*,
    $$
    where $\varphi_*\colon V\cong \mathbb{C}^{\oplus T_0}\to \mathbb{C}^{\oplus T_1}\cong V$ denotes the $\mathbb{C}$-linear map induced by $\varphi\colon T_0\to T_1$.
    We have $\det \varphi_* = 0$ if $\varphi$ is not a bijection.
    Define $\gamma\colon T_1\to T_0$ by $\gamma(i,j)=(i,n-j)$.
    Then we have $\gamma_*=-\id$ and hence $\det \gamma_* = (-1)^{\frac{(m-1)(n-1)}{2}}$.
    If $\varphi$ is bijective, then the value of $\det (\gamma\circ\varphi)_*$ is equal to the signature of the permutation $\gamma\circ \varphi$ on $T_0$.
    Combining this and the above formula, we get
    $$
        \det K = \sum_{\substack{\varphi\colon T_0\to T_1\\ \text{admissible bij.}}}
        (-1)^{\frac{(m-1)(n-1)}{2}}\sgn (\gamma\circ \varphi).
    $$
    
    Admissible bijections from $T_0$ to $T_1$ corresponds bijectively to domino tilings of $R_{m-1,n-1}$.
    Therefore it suffices to show that, for an admissible bijection $\varphi$ corresponding to a domino tiling $D$, we have
    \begin{align}\label{sgn_vs_h(D)}
        (-1)^{\frac{(m-1)(n-1)}{2}}\sgn (\gamma\circ \varphi) =
        \begin{cases}
            \sqrt{-1}^{h(D)}&(m\equiv 1 \pmod 2)\\
            (-1)^{\frac{n^2-1}{8}}\cdot \sqrt{-1}^{h(D)}&(m\equiv 0 \pmod 2).
        \end{cases}
    \end{align}

    If $D$ is the totally vertical tiling (i.e. $h(D)=0$), then the corresponding admissible bijection $\varphi\colon T_0\to T_1$ is given by
    $$
        \varphi(i,j)=\begin{cases}
            (i,j+1)&(j\equiv 1 \pmod 2)\\
            (i,j-1)&(j\equiv 0 \pmod 2).
        \end{cases}
    $$
    Therefore the permutation $\gamma\circ\varphi\colon T_0\to T_0$ is given by
    $$
        \gamma\circ \varphi(i,j)=\begin{cases}
            (i,n-j-1)&(j\equiv 1 \pmod 2)\\
            (i,n-j+1)&(j\equiv 0 \pmod 2).
        \end{cases}
    $$
    This permutation is isomorphic to a disjoint union of $m-1$ copies of the permutation
    $$
        \sigma\colon \biggl\{1,2,\dots,\dfrac{n-1}{2}\biggr\}\to \biggl\{1,2,\dots,\dfrac{n-1}{2}\biggr\};\quad x\mapsto \dfrac{n+1}{2}-x.
    $$
    If $m$ is odd, we have $\sgn(\gamma\circ \varphi)=1$.
    If $m$ is even, then we have
    $$
        \sgn(\gamma\circ \varphi)=\sgn(\sigma)=(-1)^{\frac{(n-1)(n-3)}{8}}.
    $$
    This shows that \Cref{sgn_vs_h(D)} holds for the totally vertical tiling.
    
    Now recall from \Cref{flip_transitivity} that any domino tiling of $R_{m-1,n-1}$ can be obtained from the totally vertical tiling by a finite sequence of flips.
    A flip changes the signature of $\gamma\circ\varphi$, and also changes $h(D)$ by $2$.
    This shows that \Cref{sgn_vs_h(D)} holds for any domino tiling.
\end{proof}

\subsection{Realization in an algebra}

Next we realize the Kasteleyn operator as the multiplication map in a certain finite dimensional $\mathbb{C}$-algebra.
We set $\zeta_N = e^{2\pi i/N}$ for $N\geq 1$.
First we define $\mathbb{C}$-algebras $A_m,A_m^+$ by
$$
    A_m:=\mathbb{C}[X]/((X^{2m}-1)/(X^2-1)),\quad A_m^+:=\mathbb{C}[x+x^{-1}]\subset A_m,
$$
where $x=[X]\in A_m$.
Let $P_m(T)$ be the unique polynomial of degree $m-1$ satisfying
$$
    P_m(T+T^{-1})=\dfrac{1}{T^{m-1}}\cdot\dfrac{T^{2m}-1}{T^2-1} = T^{m-1}+T^{m-3}+\cdots+T^{-(m-1)}.
$$
It is easy to see that there is an isomorphism of $\mathbb{C}$-algebras $\mathbb{C}[T]/(P_m(T))\cong A_m^+;\;t\mapsto x+x^{-1}$, where $t=[T]$.
The roots of $P_m$ are given as follows:
$$
    P_m(z)=0\iff z\in \{\zeta_{2m}^i + \zeta_{2m}^{-i}\mid 1\leq i\leq m-1\}.
$$

Similarly, we define $\mathbb{C}$-algebras $B_n,B_n^+$ by
$$
    B_n:=\mathbb{C}[Y]/((Y^n-1)/(Y-1)),\quad B_n^+:=\mathbb{C}[y+y^{-1}]\subset B_n,
$$
where $y=[Y]\in B_n$.
Let $Q_n(T)$ be the unique polynomial of degree $\frac{n-1}{2}$ satisfying
$$
    Q_n(T+T^{-1})=\dfrac{1}{T^{\frac{n-1}{2}}}\cdot\dfrac{T^n-1}{T-1} = T^{\frac{n-1}{2}}+T^{\frac{n-3}{2}}+\cdots+T^{-\frac{n-1}{2}}.
$$
Then it is easy to see that there is an isomorphism of $\mathbb{C}$-algebras $\mathbb{C}[T]/(Q_n(T))\cong B_n^+;\;t\mapsto y+y^{-1}$, where $t=[T]$.
The roots of $Q_n$ are given as follows:
$$
    Q_n(z)=0\iff z\in \{\zeta_n^j + \zeta_n^{-j}\mid 1\leq j\leq \tfrac{n-1}{2}\}.
$$

We define a Laurent polynomial $S_i(T)$ by
$$
    S_i(T):=\dfrac{T^i-T^{-i}}{T-T^{-1}} = T^{i-1}+T^{i-3}+\cdots+T^{-(i-1)}.
$$
Then $\{S_i(x)\mid 1\leq i\leq m-1\}$ is a $\mathbb{C}$-basis of $A_m^+$ , and $\{S_j(y)\mid 1\leq j\leq \frac{n-1}{2}\}$ is a $\mathbb{C}$-basis of $B_n^+$. 
As $S_j(y)=-S_{n-j}(y)$, the set $\{S_j(y)\mid 1\leq j\leq n-1,\;j\equiv \mu \pmod 2\}$ is also a $\mathbb{C}$-basis of $B_n^+$ for each $\mu\in \{0,1\}$.
Consequently, the set
$$
    \{S_i(x)S_j(y)\mid 1\leq i\leq m-1,\;1\leq j\leq n-1,\;i+j\equiv \mu\pmod 2\}
$$
is a $\mathbb{C}$-basis of $A_m^+\otimes_\mathbb{C} B_n^+$ for each $\mu\in \{0,1\}$.
Note that we have
$$
    S_0(x)=S_m(x)=0,\quad S_0(y)=S_n(y)=0.
$$
The relation $S_j(y)=-S_{n-j}(y)$ implies that there is an isomorphism of $\mathbb{C}$-vector spaces
$$
    \lambda \colon V\xrightarrow{\sim} A_m^+\otimes_\mathbb{C} B_n^+;\quad e_{i,j}\mapsto S_i(x)S_j(y).
$$

\begin{lemma}\label{detK_vs_norm}
    Let $\alpha:=x+x^{-1}+y+y^{-1}\in A_m^+\otimes_\mathbb{C} B_n^+$.
    Then the following diagram of $\mathbb{C}$-vector spaces is commutative:
    $$
    \xymatrix{
        V\ar[r]^-{\lambda}_-{\sim}\ar[d]^-{K}      &A_m^+\otimes_\mathbb{C} B_n^+\ar[d]^-{\times \alpha}\\
        V\ar[r]^-{\lambda}_-{\sim}                        &A_m^+\otimes_\mathbb{C} B_n^+.
    }
    $$
    In particular, we have $\det K=\Nm_{A_m^+\otimes_\mathbb{C} B_n^+/\mathbb{C}}(\alpha)$.
\end{lemma}

\begin{proof}
    We have
    \begin{align*}
        \alpha\cdot \lambda(e_{i,j})
        =&(x+x^{-1}+y+y^{-1})S_i(x)S_j(y)\\
        =&(x+x^{-1})S_i(x)S_j(y) + S_i(x)(y+y^{-1})S_j(y)\\
        =&(S_{i-1}(x)+S_{i+1}(x))S_j(y)+S_i(x)(S_{j-1}(y)+S_{j+1}(y))\\
        =&\lambda(e_{i-1,j})+\lambda(e_{i+1,j})+\lambda(e_{i,j-1})+\lambda(e_{i,j+1})\\
        =&\lambda(K(e_{i,j})).
    \end{align*}
    This shows that the above diagram is commutative.
\end{proof}

\subsection{Computation of the norm}

Finally, we compute the norm of $\alpha \in A_m^+\otimes_\mathbb{C} B_n^+$.

\begin{lemma}\label{norm_computation}
    We have
    $$
        \Nm_{A_m^+\otimes_\mathbb{C} B_n^+/\mathbb{C}}(\alpha)=
        \zeta_n^N
        \prod_{j=1}^{\frac{n-1}{2}}\dfrac{\xi_n^{mj}-1}{\xi_n^j-1}
    $$
    for some $N\in \mathbb{Z}$, where $\xi_n=\zeta_n^2$.
    If $m,n$ are not coprime, then this norm is $0$.
\end{lemma}

\begin{proof}
    Recall that we have isomorphisms $A_m^+\cong \mathbb{C}[T]/(P_m(T))$ and $B_n^+\cong\mathbb{C}[T]/(Q_n(T))$.
    Therefore we have $R:=A_m^+\otimes_\mathbb{C} B_n^+\cong \mathbb{C}[T_1,T_2]/(P_m(T_1),Q_n(T_2))$.
    If we write $t_i=[T_i]\in R$, then we have $\alpha=t_1+t_2$.
    Using \Cref{norm_product}, the norm can be computed as follows:
    \begin{align*}
        \Nm_{R/\mathbb{C}}(\alpha)
        =&\Nm_{R/\mathbb{C}}(t_1+t_2)\\
        =&\prod_{i=1}^{m-1}\prod_{j=1}^{\frac{n-1}{2}}
        (\zeta_{2m}^i+\zeta_{2m}^{-i}+\zeta_{n}^j+\zeta_{n}^{-j})\\
        =&\prod_{i=1}^{m-1}\prod_{j=1}^{\frac{n-1}{2}}
        \zeta_n^{-j}(\zeta_{2m}^i+\zeta_n^j)(\zeta_{2m}^{-i}+\zeta_n^{j})\\
        =&\zeta_n^N\prod_{\substack{1\leq i\leq 2m-1\\ i\neq m}}\prod_{j=1}^{\frac{n-1}{2}}(\zeta_{2m}^i+\zeta_n^j)
    \end{align*}
    for some $N\in \mathbb{Z}$.
    Using that the numbers $\zeta_{2m}^i\;(1\leq i\leq 2m-1,\;i\neq m)$ are the roots of the polynomial $(T^{2m}-1)/(T^2-1)$, we get
    $$
        \Nm_{R/\mathbb{C}}(\alpha)
        =
        \zeta_n^N\prod_{j=1}^{\frac{n-1}{2}}\dfrac{\zeta_n^{2mj}-1}{\zeta_n^{2j}-1}
        =
        \zeta_n^N\prod_{j=1}^{\frac{n-1}{2}}\dfrac{\xi_n^{mj}-1}{\xi_n^j-1}
    $$
    for some $N\in \mathbb{Z}$.
    If $m,n$ are not coprime, then this product is $0$ since $\xi_n^{mj}=1$ for some $j$.
\end{proof}

\begin{definition}
    Suppose that $m,n$ are coprime.
    Then multiplication by $m$ defines a permutation on $\mathbb{Z}/n\mathbb{Z}$.
    We set $H_n = \{1,2,\dots,\frac{n-1}{2}\}\subset \mathbb{Z}/n\mathbb{Z}$ and define a function $\varepsilon\colon H_n\to \{\pm 1\}$ so that $mH_n = \{\varepsilon(j) j\mid j\in H_n\}$.
    We write $G_n=\{j\in H_n \mid \varepsilon(j)=-1\}$.
\end{definition}

\begin{lemma}\label{product_computation}
    Suppose that $m,n$ are coprime.
    Then we have
    $$
        \displaystyle
        \prod_{j=1}^{\frac{n-1}{2}}\dfrac{\xi_n^{mj}-1}{\xi_n^j-1}=\xi_n^{N'}(-1)^{\#G_n}
    $$
    for some $N'\in \mathbb{Z}$, where $\xi_n = \zeta_n^2$.
\end{lemma}

\begin{proof}
    Using the definition of $\varepsilon$ and $G_n$, we can compute as follows:
        $$
            \prod_{j=1}^{\frac{n-1}{2}}\dfrac{\xi_n^{mj}-1}{\xi_n^j-1} =
            \prod_{j\in mH_n}(\xi_n^j-1)\prod_{j\in H_n}(\xi_n^j-1)^{-1} =
            \prod_{j\in H_n}\dfrac{\xi_n^{\varepsilon(j)j}-1}{\xi_n^j-1} =
            \prod_{j\in G_n}\dfrac{\xi_n^{-j}-1}{\xi_n^j-1} =
            \prod_{j\in G_n}-\xi_n^{-j}.
        $$
    The last term can be written as $\xi_m^{N'}(-1)^{\#G_n}$ for some $N'\in \mathbb{Z}$.
\end{proof}

The following lemma is a generalization of Gauss's lemma:

\begin{lemma}\label{Gauss_lemma}
     Suppose that $m,n$ are coprime.
     Then we have
     $$
        (-1)^{\#G_n} = \jac{m}{n}.
     $$
\end{lemma}

\begin{proof}
    We prove this by induction on $n$.
    For $n=1$, the claim is trivial.
    Let $n\geq 2$ and write $n=pn'$ where $p$ is a prime number.
    Since $n$ is odd, $p$ is also odd.
    We split $H_n$ into two parts:
    $$
        H'_n = H_n \cap p\mathbb{Z}/n\mathbb{Z}, \quad H''_n = H_n \setminus H'_n.
    $$
    Accordingly, we set $G'_n=G_n\cap H'_n$ and $G''_n=G_n\cap H''_n$.
    The canonical isomorphism $p\mathbb{Z}/n\mathbb{Z}\cong \mathbb{Z}/n'\mathbb{Z}$ identifies $H'_n$ (resp. $G'_n$) with $H_{n'}$ (resp. $G_{n'}$).
    Therefore the induction hypothesis shows that $(-1)^{\#G'_n} = \jac{m}{n'}$.
    Let $x = \prod_{i\in H''_n}i$.
    Then we have
    $$
        \prod_{i\in H''_n} mi = m^{\frac{n-n'}{2}}x\equiv (m^{\frac{p-1}{2}})^{n'}x \equiv \jac{m}{p}x \pmod p.
    $$
    On the other hand, we have $mH''_n = \{\varepsilon(i)i\mid i\in H''_n\}$, so we get
    $$
        \prod_{i\in H''_n} mi = \prod_{i\in H''_n}\varepsilon(i)i = (-1)^{\#G''_n}x.
    $$
    This shows that $(-1)^{\#G''_n} = \jac{m}{p}$.
    Combining these results, we get
    $$
        (-1)^{\#G_n} = (-1)^{\#G'_n}(-1)^{\#G''_n} = \jac{m}{p}\jac{m}{n'} = \jac{m}{n}.
    $$
\end{proof}

Now we can prove our main theorem:

\begin{proof}[Proof of \Cref{main}]
    By \Cref{kasteleyn_det}, we have
    $$
        \sum_{D\in \D(R_{m-1,n-1})} \sqrt{-1}^{h(D)} =
        \begin{cases}
            \det K & (m\equiv 1 \pmod 2)\\
            (-1)^{\frac{n^2-1}{8}}\cdot \det K & (m\equiv 0 \pmod 2).
        \end{cases}
    $$
    The second supplementary law for the Jacobi symbol shows that
    $$
        (-1)^{\frac{n^2-1}{8}} = \jac{2}{n},
    $$
    so the theorem is equivalent to $\det K=\jac{m}{n}$.
    By \Cref{detK_vs_norm}, we have $\det K=\Nm_{A_m^+\otimes_\mathbb{C} B_n^+/\mathbb{C}}(\alpha)$.
    \Cref{norm_computation}, \Cref{product_computation}, and \Cref{Gauss_lemma} shows that $\det K$ is of the form $\zeta_m^M\jac{m}{n}$ for some $M\in \mathbb{Z}$.
    Since $\det K$ is an integer, we must have $\det K=\jac{m}{n}$.
\end{proof}

\section{Combinatorial proof}
In this section, we present an elementary combinatorial proof of this theorem.
This proof is completely independent from the algebraic proof.
The idea is to decompose the board so that it is almost symmetric, and compute the left hand side of the theorem using involution.

\begin{definition}
Let $X$ be a board.
We define $s(D)=\sqrt{-1}^{h(D)}$ for $D \in \D(X)$, and set
$$
    S(X)=\sum_{D \in \D(X)} {s(D)}.
$$
\end{definition}

\begin{definition}
For positive integers $(a,b)$, we define $L(a,b)\subset\mathbb{Z}\times \mathbb{Z}$ by
$$
    L(a,b)=([a]\times [1]) \cup ([1] \times [b])
$$
We define $L(a,b)=\emptyset$ if $a$ or $b$ is $0$. 

For non-negative integers $a_1,\ldots, a_n$ and $b_1,\ldots, b_n$, we define a subset $L(a_1,\ldots a_n;b_1, \ldots, b_n)\subset \mathbb{Z}\times\mathbb{Z}$ as follows:
$$
    L(a_1,\ldots a_n;b_1, \ldots, b_n) = \bigcup_{k=1}^n ((k-1,k-1)+L(a_k,b_k)).
$$
\end{definition} 

Informally, $L(a,b)$ describes an L-shaped board, which has $a$ for width and $b$ for height, and $L(a_1,\ldots, a_n; b_1,\ldots, b_n)$ is structured by placing $L(a_1,b_1), L(a_2,b_2),\ldots, L(a_n,b_n)$ in the diagonal of $\mathbb{Z}_{>0}\times\mathbb{Z}_{>0}$ in this order. 

The next lemma is the key in our proof; this computes $S(X)$ for an almost symmetric board $X$.

\begin{lemma}\label{sum_for_symmetric_board}
    Let $a_1,\ldots, a_n, b_1,\ldots, b_n$ be non-negative integers.
    If $\lvert a_k-b_k \rvert \leq 1$ for every $k$, then 
    $$ S(L(a_1,\ldots, a_n;b_1, \ldots, b_n))=
    \begin{cases}
        0 & (a_k=b_k\neq 0 \text{ for some } k), \\
        \prod_{k=1}^n \sqrt{-1}^{\lfloor{a_k/2}\rfloor} & (\text{otherwise}).\\
    \end{cases}
    $$
\end{lemma}

\begin{proof}
    For simplicity, we write $L$ for $L(a_1,\cdots,a_n;b_1,\cdots,b_n)$. 
    Let $D$ be any domino tiling of $L$.
    We call the square $(i,j)$  \emph{violating} if and only if
    \begin{enumerate}
        \item $i>j$ and $\{(i,j),(i,j+1)\}\in D$, or
        \item $i<j$ and $\{(i,j),(i+1,j)\}\in D$. 
    \end{enumerate}
    We define a preorder on $\mathbb{Z}_{>0}\times\mathbb{Z}_{>0}$ as follows:
    we set $(i,j)\leq (k,l)$ if and only if
    \begin{enumerate}
        \item $i+j < k+l$ , or
        \item $i+j = k+l$ and $\lvert{j-i}\rvert \leq \lvert{k-l}\rvert$. 
    \end{enumerate}
    Write $(i,j)\sim(k,l)$ if $(i,j)\leq (k,l)$ and $(k,l) \leq (i,j)$; this is equivalent to $(i,j)\in \{(k,l),(l,k)\}$.
    For example, we have
    $$
        (1,1)<(1,2)\sim(2,1)<(2,2)<(1,3)\sim(3,1)<(2,3)\sim(3,2)<(1,4)\sim(4,1)<(3,3)\cdots.
    $$ 
    
    Let $\D_{i,j} (\subset \D(L))$ be the set of domino tilings of $L$ whose minimal violating square (in the order defined earlier) is $(i,j)$.
    We will prove that $\D_{i,j}\cap \D_{j,i}=\emptyset$, and that there exists a bijection $\iota:\D_{i,j} \to \D_{j,i}$ which satisfies $h(\iota(D))=h(D)\pm 2$.
    
    Let $D\in \D_{i,j}$ where $i>j$.
    We first assume that $i-j$ is even.
    As a result of the minimality of the violating square, we have $\{(i-1,j),(i-2,j)\}\in D$.
    If $i-3>j$, then we also have $\{(i-3,j),(i-4,j)\}\in D$.
    Repeating this argument, we see that
    $$
        \{(x,j),(x+1,j)\}\in D\quad (j\leq x\leq i-2,\;x\equiv i\pmod 2).
    $$
    The same argument shows that
    $$
        \{(x,j+1),(x+1,j+1)\}\in D\quad (j+2\leq x\leq i-2,\;x\equiv i\pmod 2).
    $$
    
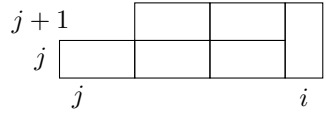
\begin{center}
    \begin{tikzpicture}[scale = 0.5]
    \newcommand{\stackedrectangles}{
        \draw (0,0) rectangle (2,1);  
        \draw (0,1) rectangle (2,2);  
    }

    \draw (0,0) rectangle (2,1);
    \node at (0.5,-0.5) {$j$};
    \node at (-0.5,0.5) {$j$};
    \node at (-0.5,1.5) {$j+1$};
    \begin{scope}[shift={(2,0)}]
        \stackedrectangles
    \end{scope}
    \begin{scope}[shift={(4,0)}]
        \stackedrectangles
    \end{scope}
    \begin{scope}[shift={(6,0)}]
        \draw (0,0) rectangle (1,2);
        \node at (0.5,-0.5) {$i$};
    \end{scope}
    \end{tikzpicture}
    \end{center}
    
    Now, again by the minimality of the violating square, we have $\{(j,j+1), (j,j+2)\}\in D$.
    If $i>j+3$, then we also have $\{(j,j+3), (j,j+4)\}$.
    Repeating this argument, we see that
    $$
        \{(j,y),(j,y+1)\}\in D\quad (j+1\leq y\leq i-1,\;y\not\equiv i\pmod 2).
    $$
    The same argument shows that
    $$
        \{(j+1,y),(j+1,y+1)\}\in D\quad (j+1\leq y\leq i-1,\;y\not\equiv i\pmod 2),
    $$
    noting that we have $(j+1,i-1)<(i,j)$.
    Here, we used the condition $|a_i-b_i|, |a_{i+1}-b_{i+1}|\leq 1$ to ensure that $(j,i-1),(j+1,i-1)$ are in $L$. This proves $D\not \in \D_{j,i}$ and hence $\D_{j,i}\cap\D_{i,j}=\emptyset$. 

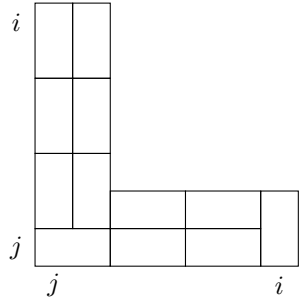
    \begin{center}
    \begin{tikzpicture}[scale = 0.5]
    \newcommand{\stackedrectangles}{
        \draw (0,0) rectangle (2,1);  
        \draw (0,1) rectangle (2,2);  
    }
    \newcommand{\alignedrectangles}{
        \draw (0,0) rectangle (1,2);  
        \draw (1,0) rectangle (2,2);  
    }

    \draw (0,0) rectangle (2,1);
    \node at (0.5,-0.5) {$j$};
    \node at (-0.5,0.5) {$j$};
    \begin{scope}[shift={(2,0)}]
        \stackedrectangles
    \end{scope}
    \begin{scope}[shift={(4,0)}]
        \stackedrectangles
    \end{scope}
    \begin{scope}[shift={(6,0)}]
        \draw (0,0) rectangle (1,2);
        \node at (0.5,-0.5) {$i$};
    \end{scope}
    \begin{scope}[shift={(0,1)}]
        \alignedrectangles
    \end{scope}
    \begin{scope}[shift={(0,3)}]
        \alignedrectangles
    \end{scope}
    \begin{scope}[shift={(0,5)}]
        \alignedrectangles
        \node at (-0.5,1.5) {$i$};
    \end{scope}
    \end{tikzpicture}
    \end{center}

    When $i-j$ is odd, a similar argument shows that
    \begin{align*}
        \{(x,j),(x+1,j)\}\in D\quad &(j+1\leq x\leq i-2,\;x\equiv i\pmod 2),\\
        \{(x,j+1),(x+1,j+1)\}\in D\quad &(j+1\leq x\leq i-2,\;x\equiv i\pmod 2),\\
        \{(j,y),(j,y+1)\in D\quad &(j\leq y\leq i-1,\;y\not \equiv i\pmod 2),\\
        \{(j+1,y),(j+1,y+1)\in D\quad &(j+2\leq y\leq i-1,\;y\not \equiv i\pmod 2).
    \end{align*}
    This proves $D\not \in \D_{j,i}$ and hence $\D_{j,i}\cap\D_{i,j}=\emptyset$.

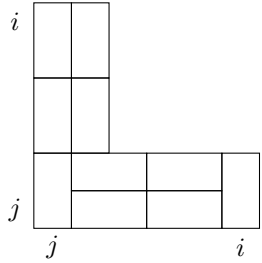
    \begin{center}
    \begin{tikzpicture}[scale = 0.5]
    \newcommand{\stackedrectangles}{
        \draw (0,0) rectangle (2,1);  
        \draw (0,1) rectangle (2,2);  
    }
    \newcommand{\alignedrectangles}{
        \draw (0,0) rectangle (1,2);  
        \draw (1,0) rectangle (2,2);  
    }

    \begin{scope}[shift={(1,0)}]
        \stackedrectangles
    \end{scope}
    \node at (0.5,-0.5) {$j$};
    \node at (-0.5,0.5) {$j$};
    \begin{scope}[shift={(3,0)}]
        \stackedrectangles
    \end{scope}
    \begin{scope}[shift={(5,0)}]
        \draw (0,0) rectangle (1,2);
        \node at (0.5,-0.5) {$i$};
    \end{scope}
    \begin{scope}[shift={(0,0)}]
        \draw (0,0) rectangle (1,2);
    \end{scope}
    \begin{scope}[shift={(0,2)}]
        \alignedrectangles
    \end{scope}
    \begin{scope}[shift={(0,4)}]
        \alignedrectangles
        \node at (-0.5,1.5) {$i$};
    \end{scope}
    \end{tikzpicture}
    \end{center}

    Next we define a bijection $\iota\colon \D_{i,j}\to \D_{j,i}$.
    For $i>j$, we define $\iota\colon \D_{i,j}\to \D_{j,i}$ by transposing the L-shaped region we have drawn above.
    We also define $\iota\colon \D_{j,i}\to \D_{i,j}$ symmetrically.
    By definition, we have $\iota\circ\iota = \id$, so $\iota$ is a bijection.
    Suppose that $D\in \D_{i,j}$.
    By the transposing procedure, $i-j+1$ vertical dominoes will become horizontal, and $i-j-1$ horizontal dominoes will become vertical.
    Therefore we have $h(\iota(D))=h(D)\pm2$. 

    Let $\D'(L)=\D(L)\setminus \bigcup_{i,j} \D_{i,j}$.
    The argument above yields $S(L)=\sum_{D\in \D'(L)} s(D)$. 
    If $D\in \D'(L)$, then for all $k$, there is no domino in $D$ which intersects with both $(k-1,k-1)+L(a_k,b_k)$ and $(k,k)+L(a_{k+1},b_{k+1})$. 
    Therefore we have $S(L)=\prod_{k=1}^n S(L(a_k,b_k))$.
    A domino tiling of $L(a_k,b_k)$ exists if and only if $\lvert a_k-b_k \rvert =1$ or $a_k=b_k=0$, and in this case, there are $\lfloor a_k/2 \rfloor$ horizontal dominoes. This completes the proof.
\end{proof}

Next we want to divide the board to use \Cref{sum_for_symmetric_board}.
For this, we need some definition. 

\begin{definition}
    Let $X$ be a board, $T$ be a subset of $X$, and $D$ be a domino tiling of $X$. We say that $T$ is \emph{closed} for $D$ if there is no domino which intersects with both $T$ and $X\setminus T$. 
    Also, we define $\Cl_D(T)$ as a minimal closed subset for $D$ containing $T$.
    We write $\Cl(T)$ for the union of all $\Cl_D(T)$ for $D\in D_X$. 
    We define $S(T;X)=\sum_{D\in \D(X): \Cl_D(T)=X} s(D)$. 
\end{definition}
The next lemma is almost straightforward by definition.
\begin{lemma}\label{decomposition_of_board}
    Let $X$ be a board and $T$ be a subset of $X$. 
    Then the following equality holds:
    $$
        S(X)=\sum_{U:T\subset U\subset \Cl(T)}S(X\setminus U)S(T;U).
    $$
\end{lemma}
\begin{proof}
    We have
    \begin{align*}
        \sum_{D\in \D(X)}{s(D)}&{}=\sum_{U:T\subset U\subset \Cl(T)}{\biggl(\sum_{D\in \D(X):\Cl_D(T)=U}{s(D)}\biggr)}\\
        &{}=\sum_{U:T\subset U\subset \Cl(T)}\biggl(S(X\setminus U)\sum_{D\in \D(U):\Cl_D(T)=U} s(D)\biggr)\\
        &{}=\sum_{U:T\subset U\subset \Cl(T)}S(X\setminus U)S(T;U).
    \end{align*}
\end{proof}

Using the lemmas above, we can prove the periodicity of $S(R_{\bullet,n})$. 
\begin{proposition}\label{periodic_of_S}
For positive integers $m,n$, we have
$$
    S(R_{m+n+1,n})=S(R_{m,n})\cdot 
    \begin{cases}
    \sqrt{-1}^{\frac{n^2+2n}{4}} & (n\equiv 0\pmod{2}), \\
    \sqrt{-1}^{\frac{n^2+2n+1}{4}} & (n\equiv 1\pmod{2}).
    \end{cases}
$$
\end{proposition}

\begin{proof}
    Let $T=R_{n,n}$ and $T^{+}_A=T\cup (\{n+1\}\times A)$, $T^{-}_A=T\setminus (\{n\}\times A)$. 
    By \Cref{decomposition_of_board}, we have
    $$
        S(R_{m+n+1,n})=\sum_{A:A\subset [n]} S(R_{m+n+1,n}\setminus T^{+}_A)S(T;T^{+}_A).
    $$
    Let $D$ be any domino tiling of $T^{+}_A$ satisfying $\Cl_T(D)=T^{+}_A$. 
    Then for all $a\in A$, we have $\{(m,a),(m+1,a)\}\in D$.
    Therefore we have $S(T;T^{+}_A)=\sqrt{-1}^{\#A}S(T^{-}_A)$.
    We can write
    $$
        T^{-}_A=L(n-\chi_A(1),n-1-\chi_A(2),\ldots ,1-\chi_A(n);n,n-1,\ldots,1)
    $$
    where $\chi_A$ denotes the characteristic function for $A$.
    Therefore we can use \Cref{sum_for_symmetric_board} to compute $S(T^{-}_A)$.
    If $A\neq [n]$, then $S(T^{-}_A)=0$ .
    If $A=[n]$, then $S(T^{-}_A)=\sqrt{-1}^{f(n)}$, where $f(n)=\sum_{i=1}^n \lfloor\frac{n-i}{2}\rfloor$. An easy computation shows that
    $f(n)=\frac{n^2-2n}{4}$ for even $n$, and $f(n)=\frac{n^2-2n+1}{4}$ for odd $n$. 
    Therefore we get
    $$
    S(R_{m+n+1,n})=S(R\setminus T^{+}_{[n]})\cdot \sqrt{-1}^nS(T^{-}_{[n]})=    
    S(R_{m,n})\cdot
    \begin{cases}
    \sqrt{-1}^{\frac{n^2+2n}{4}} & (n\equiv 0\pmod{2}), \\
    \sqrt{-1}^{\frac{n^2+2n+1}{4}} & (n\equiv 1\pmod{2}).
    \end{cases}
    $$
\end{proof}

\begin{corollary}\label{not_coprime_implies_zero}
If $m,n$ are not coprime, then $S(R_{m-1,n-1})=0$. 
\end{corollary}
\begin{proof}
    We proceed by induction on $\max(m,n)$. By transposing, we may assume that $m\geq n$. If $m>n$, then using \Cref{periodic_of_S} and the induction hypothesis, we deduce that $S(R_{m-1,n-1})=\sqrt{-1}^{\bullet}S(R_{m-n-1,n-1})=0$. If $m=n$, then the claim follows from \Cref{sum_for_symmetric_board}.
\end{proof}

If we use the reciprocity theorem for the Jacobi symbol, then the same induction proves our main theorem. However, by dividing the board in some intricate way, we can prove our main theorem without using the quadratic reciprocity, which yields a new proof of the reciprocity.

\

From now on, we fix positive odd integers  $m,n$ satisfying $m>n$.
\begin{definition}
    We define $Y=\{(i,j)\in R_{m-1,n-1}\mid i+j < \frac{m+n}{2}\}$. 
    For any $A\subset [n-1]$, we set
    $Y_A = Y\cup \{(\tfrac{m+n}{2}-a,a)|a\in A\}$. 
\end{definition}

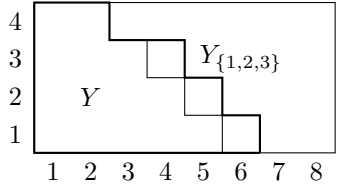
\begin{center}
    \begin{tikzpicture}[scale = 0.5]
    \draw (0,0) rectangle (8,4);
    \draw (5,0) -- (5,1) -- (4,1) -- (4,2) -- (3,2) -- (3,3) -- (2,3) -- (2,4);
    \draw [thick] (6,0) -- (6,1) -- (5,1) -- (5,2) -- (4,2) -- (4,3) -- (3,3) -- (2,3) -- (2,4) -- (0,4) -- (0,0) -- (6,0);
    \node at (1.5,1.5) {$Y$};
    \node at (5.5,2.5) {$Y_{\{1,2,3\}}$};
    \node foreach \x in {1,2,3,4,5,6,7,8} at (\x-0.5,-0.5) {$\x$};
    \node foreach \y in {1,2,3,4} at (-0.5,\y-0.5) {$\y$};
    \end{tikzpicture}
\end{center}

\begin{proposition}\label{decomposition_to_Y}
    We have
    $$
        S(R_{m-1,n-1})=\sum_{A:A\subset[n-1]} S(Y_{n-A^c})S(Y_A).
    $$
    Here, we write $A^c=[n-1]\setminus A$ and $n-B=\{n-i\mid i\in B\}$. 
\end{proposition}

\begin{proof}
    We write $R=R_{m-1,n-1}$.
    Applying \Cref{decomposition_of_board} for $R$ and $Y$, we get
    $$
        S(R)=\sum_{A:A\subset[n-1]} S(R\setminus Y_A)S(Y;Y_A).
    $$
    As there is no pair of squares in $Y_A\setminus Y$ which are adjacent, we have $S(Y;Y_A)=S(Y_A)$.
    We also have
    $$
        R\setminus Y_A=\{(i,j)\in R\mid i+j > \tfrac{m+n}{2}\} \cup \{(\tfrac{m+n}{2}-a,a)\mid a\in A^c\},
    $$
    so the congruence transformation $(i,j)\mapsto(m-i,n-j)$ identifies $R\setminus Y_A$ with $Y_{n-A^c}$. 
    This induces a bijection $\D(R\setminus Y_A)\xrightarrow{\sim} \D(Y_{n-A^c})$ which does not change the number of horizontal dominoes, so we get $S(R\setminus Y_A)=S(Y_{n-A^c})$. 
\end{proof}


The next lemma shows that $S(Y_A)$ is $0$ for almost all $A$. 

\begin{lemma}\label{sum_for_half_board}
Suppose that $n<m<3n$.
Then we have $S(Y_A)\in \{0,\pm1,\pm i\}$.
Moreover, we have $S(Y_A)\neq 0$ if and only if $A$ satisfies all of the following conditions:
\begin{itemize}
    \item $\frac{m-n}{2}\in A$.
    \item If $i,j\in [n-1]$, $i+j\equiv \frac{m}{2}\pmod n$, and $i\neq j$, then $\#(\{i,j\}\cap A)= 1$.
    \item If $i\in [n-1]$ and $2i\equiv \frac{m}{2}\pmod n$, then $i\not\in A$.
\end{itemize}
\end{lemma}
\begin{proof}
    Let $Z=Y_A\cap R_{n-1,n-1}$.
    First we assume that $\frac{m-n}{2}\not\in A$.
    For $B\subset[\frac{m-n}{2}-1]$, let $Z^{+}_B=Z\cup (\{n\}\times B)$ and $Z^{-}_B=Z\setminus (\{n-1\}\times B)$. 
    Using \Cref{decomposition_of_board}, we get
    $$
        S(Y_A)=\sum_{B:B\subset[\frac{m-n}{2}-1]}S(Y_A\setminus Z^{+}_B)S(Z;Z^{+}_B).
    $$
    As in the proof of \Cref{periodic_of_S}, we have $S(Z;Z^{+}_B)=S(Z^{-}_B)$.
    However, we can write $Z^{-}_B=L(a_1,a_2\cdots,a_k;b_1,b_2,\ldots b_k)$ with $a_{\frac{m-n}{2}}=b_{\frac{m-n}{2}}=\frac{3n-m}{2}>0$.
    Therefore we get $S(Y_A)=0$ by \Cref{sum_for_symmetric_board}.

    Next we assume that $\frac{m-n}{2}\in A$.
    For $B\subset[\frac{m-n}{2}]$, as before, we set $Z^{+}_B=Z\cup (\{n\}\times B)$ and $Z^{-}_B=Z\setminus (\{n-1\}\times B)$. 
    Using \Cref{decomposition_of_board}, we get
    $$
        S(Y_A)=\sum_{B:B\subset[\frac{m-n}{2}]}S(Y_A\setminus Z^{+}_B)S(Z;Z^{+}_B).
    $$
    As in the proof of \Cref{periodic_of_S}, we have $S(Z;Z^{+}_B)=S(Z^{-}_B)$.
    We can write
    $$Z^{-}_B=L(a_1,\ldots,a_k;b_1,\ldots b_k),$$
    where $k=\lfloor \frac{m+n}{4}\rfloor$ and
    $$
        |a_i-b_i|=
        \begin{cases}
            \chi_B(i) & (i\leq  \frac{m-n}{2}) \\
            |\chi_A(i)-\chi_A(\tfrac{m+n}{2}-i)| & (i>\frac{m-n}{2}).
        \end{cases}
    $$
    By \Cref{sum_for_symmetric_board}, $S(Z_B^{-})$ is non-zero if and only if all of the following conditions are true:
    \begin{itemize}
        \item $B=[\frac{m-n}{2}]$.
        \item for any $i,j\in [n-1]$ satisfying $i+j=\tfrac{m+n}{2}$ and $i\neq j$, we have $\#(\{i,j\}\cap A)= 1$.
        \item $\tfrac{m+n}{4}\not\in A$.
    \end{itemize}
    In particular, we have
    $$
        S(Y_A)=S(Y_A\setminus Z^{+}_{[\frac{m-n}{2}]})S(Z^{-}_{[\frac{m-n}{2}]})
    $$
    We can write
    $Y_A\setminus Z^{+}_{[\frac{m-n}{2}]}=(n,0)+L(c_1,\ldots c_l;d_1,\ldots d_l)$,
    where $l=\lfloor \frac{m-n}{4} \rfloor$ and 
    $$
        |c_i-d_i|=|\chi_A(i)-\chi_A(\tfrac{m-n}{2}-i)|.
    $$
    By \Cref{sum_for_symmetric_board}, $S(Y_A\setminus Z^{+}_B)$ is non-zero if and only if the following conditions are true:
    \begin{itemize}
        \item for any $i,j\in [n-1]$ satisfying $i+j=\tfrac{m-n}{2}$ and $i\neq j$, we have $\#(\{i,j\}\cap A)= 1$.
        \item $\frac{m-n}{4}\not\in A$.
    \end{itemize}
    This completes the proof. 
\end{proof}

Computing $S(Y_A)$ carefully, we may obtain the exact value of $S(Y_A)$. However, we only need the next lemma to prove our main theorem.

\begin{lemma}\label{parity_of_h}
Let $D$ be any domino tiling of $Y_A$. Then we have
$$h(D)\equiv \tfrac{n-1}{4}-\tfrac{\#A}{2}+\#(A\cap (2\mathbb{Z}+1))\pmod{2}.$$
\end{lemma}
\begin{proof}
    We prove this by painting $Y_A$.
    We paint black for an odd row, and white for an even row.
    Consider placing dominoes one by one.
    We keep track of the following value:
    $$
        X=\frac{1}{2}(\#\{\text{black squares not covered by dominoes}\}-\#\{\text{white squares not covered by dominoes}\}).
    $$
    When no domino is placed, we have $X = \tfrac{n-1}{4}-\tfrac{\#A}{2}+\#(A\cap (2\mathbb{Z}+1))$. Placing a vertical domino does not change $X$, and placing a horizontal domino changes $X$ by $1$. When all dominoes are placed, we must have $X=0$. This proves the claim.
\end{proof}

Next lemma connects combinatorial objects with number-theoretic objects.

\begin{lemma}\label{condition_of_A}
    Let $m,n$ be coprime odd integers satisfying $n<m<3n$.
    We have $S(Y_A)S(Y_{n-A^c})\neq 0$ if and only if
    $A=\{\overline{\frac{m}{2}},\overline{\frac{2m}{2}},\overline{\frac{3m}{2}},\ldots,\overline{\frac{n-1}{2}\cdot\frac{m}{2}}\}$. 
    Here, $\overline{i}$ denotes the residue of $i$ modulo $n$. 
\end{lemma}

\begin{proof}
    Suppose that $A$ and $n-A^c$ satisfy all conditions in \Cref{sum_for_half_board}. 
    Using the first condition, we get $\overline{\frac{m}{2}} \in A,\;n-A^c$.
    Therefore $\overline{-\frac{m}{2}} \not\in A,\;n-A^c$.
    Using the second condition for $(i,j)=(\overline{-\frac{m}{2}},\overline{\frac{2m}{2}})$, we get $\overline{\frac{2m}{2}}\in A,\;n-A^c$.
    Therefore $\overline{-\frac{2m}{2}} \not\in A,\;n-A^c$.
    Repeating this argument, we get $A=\{\overline{\frac{m}{2}},\overline{\frac{2m}{2}},\overline{\frac{3m}{2}},\ldots,\overline{\frac{n-1}{2}\cdot\frac{m}{2}}\}$. It is easy to check that this satisfies all conditions in \Cref{sum_for_half_board}. 
\end{proof}


\begin{definition}
    Suppose that $t$ is a positive integer coprime to $n$.
    Then multiplication by $t$ defines a permutation on $\mathbb{Z}/n\mathbb{Z}$.
    We set $H_n = \{2,4,\dots,n-1\}\subset \mathbb{Z}/n\mathbb{Z}$ and define a function $\varepsilon_t\colon H_n\to \{\pm 1\}$ so that $tH_n = \{\varepsilon_t(i) i\mid i\in H_n\}$.
    We write $G_{n,t}=\{i\in H_n \mid \varepsilon_t(i)=-1\}$.
\end{definition}

\begin{lemma}\label{Gauss_lemma_2}
     Suppose that $t$ is a positive integer coprime to $n$.
     Then we have
     $$
        (-1)^{\#G_{n,t}} = \jac{t}{n}.
     $$
\end{lemma}

\begin{proof}
    This is essentially the same as \Cref{Gauss_lemma}, so we omit the proof. 
\end{proof}


\begin{proof}[Proof of \Cref{main}]
    Let $m,n$ be positive integers such that $n$ is odd. 
    We need to prove
    $$
        S(R_{m-1,n-1})=        
        \begin{cases}
            \jac{m}{n}&(m\equiv 1\pmod 2)\\
            \jac{m/2}{n}&(m\equiv 0\pmod 2).
        \end{cases}
    $$
    By \Cref{periodic_of_S}, $S(R_{m+n-1,n-1})=S(R_{m-1,n-1})\cdot \sqrt{-1}^{\frac{n^2-1}{4}}=\jac{2}{n}S(R_{m-1,n-1})$.
    Therefore it is enough to prove when $m,n$ are odd and $n\leq m \leq 3n$. By \Cref{not_coprime_implies_zero}, we can assume $(m,n)=1$.
    By \Cref{decomposition_to_Y} and \Cref{condition_of_A}, we have $S(R_{m-1,n-1})=S(Y_A)^2$, where $A=\{\overline{\frac{m}{2}},\overline{\frac{2m}{2}},\overline{\frac{3m}{2}},\ldots,\overline{\frac{n-1}{2}\cdot\frac{m}{2}}\}$.
    By \Cref{parity_of_h}, this is equal to $(-1)^{\#G_{n,t}}$, where $t=\overline{\frac{m}{4}}$.
    By \Cref{Gauss_lemma_2}, this is equal to $\jac{t}{n}=\jac{m}{n}$.
\end{proof}

\printbibliography

\end{document}